\documentclass{amsart}
\usepackage{mathrsfs}
\usepackage{amssymb,latexsym,amsmath,amsthm,enumitem,mathrsfs}
\usepackage{hyperref}

\usepackage{color}
\usepackage{stmaryrd}
\usepackage{mathrsfs}
\usepackage{lineno}
\usepackage{bbm}
\usepackage{scalerel,stackengine}
\usepackage{tikz}

\usepackage{ esint }
\usepackage{graphicx}
\usepackage{amssymb}
\newtheorem{theorem}{Theorem}
\newtheorem{lemma}[theorem]{Lemma}
\newtheorem{remark}[theorem]{Remark}

\newtheorem{definition}[theorem]{Definition}

\newtheorem{conjecture}{Conjecture}

\theoremstyle{definition}

\newcommand{\cB}{\mathcal{B}}

\newcommand{\cD}{\mathcal{D}}

\newcommand{\cG}{\mathcal{G}}
\newcommand{\cH}{\mathcal{H}}

\newcommand{\cN}{\mathcal{N}}

\newcommand{\cT}{\mathcal{T}}
\newcommand{\cU}{\mathcal{U}}
\newcommand{\cV}{\mathcal{V}}
\newcommand{\cW}{\mathcal{W}}

\newcommand{\e}{\varepsilon}

\newcommand{\cQ}{\mathcal{Q}}

\def\set4{\mathcal I}
\def\tup14{(1,2,3,4)}

\newcommand{\supp}{\mathrm{supp}}

\newcommand{\R}{\mathbb{R}}

\newcommand{\de}{\delta} 

\newcommand{\wt}{\widetilde}

\newcommand{\dist}{\textup{dist}}

\newcommand{\BL}{\textup{BL}}
\newcommand{\dir}{\textup{dir}}
\newcommand{\tx}{\textup}

\begin{document}

 \author{Shengwen Gan}
 \address{Department of Mathematics\\
 Massachusetts Institute of Technology\\
 Cambridge, MA 02142-4307, USA}
 \email{shengwen@mit.edu}


\date{}

\title{Hausdorff dimension of unions of $k$-planes}
\maketitle

\begin{abstract}
We prove a conjecture of H\'era on the dimension of unions of $k$-planes. Let $0<k \le d<n$ be integers, and $\beta\in[0,k+1)$. If $\cV\subset A(k,n)$, with $\textup{dim}(\cV)=(k+1)(d-k)+\beta$, then $\textup{dim}(\bigcup_{V\in\cV}V)\ge d+\min\{1,\beta\}$. The proof combines a recent idea of Zahl and the Brascamp-Lieb inequality.
\end{abstract}

\section{Introduction}

In \cite{hera2018hausdorff}, H\'era made the following conjecture (see Conjecture 1.16 in \cite{hera2018hausdorff}). We slightly change the notation and state it in the following equivalent way.

\begin{conjecture}\label{conj1}
Let $0<k\le d<n$ be integers and $\beta\in [0,k+1]$. Let $\cV\subset A(k,n)$ be a set of $k$-planes in $\R^n$, with $\dim(\cV)= (k+1)(d-k)+\beta$. Here, $A(k,n)$ denotes the affine Grassmannian in $\R^n$ and $\dim(\cV)$ means the Hausdorff dimension of $\cV$ as a subset of the Riemannian manifold $A(k,n)$. Then,
\begin{equation}\label{mainbound}
    \dim(\bigcup_{V\in\cV}V)\ge d+\min\{1,\beta\}.
\end{equation}  

\end{conjecture}

There are many works related to Conjecture \ref{conj1}. See for example \cite{hera2018hausdorff}, \cite{hera2019hausdorff}, \cite{oberlin2011exceptional}, \cite{oberlin2016unions} and \cite{zahl2022unions}. We refer to \cite{zahl2022unions} for a detailed introduction on this problem.

\subsection{Sharp examples}

We first give some examples to show \eqref{mainbound} is sharp, i.e., we will find $\cV\subset A(k,n)$ with $\dim(\cV)=(k+1)(d-k)+\beta$, so that
\[ \dim(\bigcup_{V\in\cV}V)\le d+\min\{1,\beta\}. \]
Our examples $V\in\cV$ will lie in $\R^{d+1}(\subset \R^n)$.

\begin{itemize}
    \item When $\beta\in[1,k+1]$, let $\cV$ be a $((k+1)(d-k)+\beta)$-dimensional subset of $A(k,d+1)$. This is allowable since $\dim(A(k,d+1))=(k+1)(d+1-k)\ge (k+1)(d-k)+\beta$. Since every $V\in \cV$ is contained in $\R^{d+1}$, we have $\dim(\bigcup_{V\in\cV}V)\le \dim(\R^{d+1})=d+1$.

    \item When $\beta\in [0,1]$, we choose $I\subset \R$ to be a $\beta$-dimensional set. Let  
    \[ \cV=\cup_{a\in I} A(k,\R^d+(0,\dots,0,a)). \]
    Here, $A(k,\R^d+(0,\dots,0,a))$ is the set of $k$-planes that are contained in the space $\R^d+(0,\dots,0,a)$. We see that $\dim(\cV)=\dim I+\dim A(k,d)=(k+1)(d-k)+\beta$, and $\dim(\bigcup_{V\in\cV}V)=\dim(\R^d\times I)=d+\beta$.
\end{itemize}

\medskip

\subsection{Notation}

We introduce some notation that is needed for the later discussions. 

Let $G(k,n)$ be the set of $k$-dimensional subspaces in $\R^n$. For every $k$-plane $V$, we can uniquely write it as
\[ V=\dir(V)+x_V, \]
where $\dir(V)\in G(k,n)$ and $x_V\in V^\perp$. $\dir(V)$ refers to the direction of $V$, as can be seen that $\dir(V)=\dir(V')\Leftrightarrow V\parallel V'$.

In this paper, we use $A(k,n)$ to denote the set of $k$-planes $V$ such that $x_V\in B^n(0,1/2)$. ($B^n(0,1/2)$ is the ball of radius $1/2$ centered at the origin in $\R^n$.)
\[ A(k,n)=\{V: V\tx{~is~a~}k\tx{~dimensional~plane~}, x_V\in B^n(0,1/2)\}. \]
Usually $A(k,n)$ denotes all the $k$-planes in other references, but for our purpose we only care about those $V$ that lie near the origin.

Next, we discuss the metrics on $G(k,n)$ and $A(k,n)$.
For $V_1, V_2\in G(k,n)$, we define
\[ d(V_1,V_2)=\|\pi_{V_1}-\pi_{V_2}\|. \]
Here, $\pi_{V_1}:\R^n\rightarrow V_1$ is the orthogonal projection. We have another characterization for this metric. Define $\rho(V_1,V_2)$ to be the smallest number $\rho$ such that $B^n(0,1)\cap V_1\subset N_{\rho}(V_2)$. We have the comparability of $d(\cdot,\cdot)$ and $\rho(\cdot,\cdot)$.

\begin{lemma}
    There exists a constant $C>0$ (depending on $k,n$) such that
    \[ \rho(V_1,V_2)\le  d(V_1,V_2)\le  C\rho(V_1,V_2).  \]
\end{lemma}
\begin{proof}
    Suppose $B^n(0,1)\cap V_1\subset N_\rho(V_2)$, then for any $v\in\R^n$, we have
    \[ |\pi_{V_1}(v)-\pi_{V_2}(v)|\lesssim \rho|v|, \]
    which implies $d(V_1,V_2)\lesssim \rho$. On the other hand, if for any $|v|\le 1$ we have
    \[ |\pi_{V_1}(v)-\pi_{V_2}(v)|\le d|v|, \]
    then we obtain that $\pi_{V_1}(v)\subset N_{d}(V_2)$. Letting $v$ ranging over $B^n(0,1)\cap V_1$, we get $B^n(0,1)\cap V_1\subset N_{d}(V_2)$, which means $\rho(V_1,V_2)\le  d$. 
\end{proof}

We can also define the metric on $A(k,n)$ given by
\begin{equation}\label{defdist}
    d(V,V')=d(\dir(V),\dir(V'))+|x_{V}-x_{V'}|. 
\end{equation} 
Here, we still use $d$ to denote the metric on $A(k,n)$ and it will not make any confusion.

Similarly, for $V,V'\in A(k,n)$ we can define $\rho(V,V')$ to be the smallest number $\rho$ such that $B^n(0,1)\cap V\subset N_\rho(V')$. We also have the following lemma. We left the proof to the interested readers. 

\begin{lemma}\label{comparablelem}
    There exists a constant $C>0$ (depending on $k,n$) such that for $V,V'\in A(k,n)$,
    \[ C^{-1} d(V,V')\le \rho(V,V')\le C d(V,V').  \]
\end{lemma}

\begin{definition}
    For $V\in A(k,n)$ and $0<r<1$, we define 
    \[V_r:=N_r(V)\cap B^n(0,1).  \]
\end{definition}
We see that $V_r$ is morally a slab of dimensions $\underbrace{r\times \dots\times r}_{k \textup{~times}}\times \underbrace{1\times \dots\times 1}_{n-k \tx{~times}}$. We usually call $V_r$ a $k$-dimensional $r$-slab. When $k$ is already clear, we simply call $V_r$ an $r$-slab.
If $W$ is a convex set such that $C^{-1} W\subset V_r\subset C W$, then we also call $W$ an $r$-slab. Here, the constant $C$ will be a fixed large constant.

\begin{definition}
    For two $r$-slab $V_r$ and $V'_r$. We say they are comparable if $C^{-1}V_r\subset V'_r\subset C V_r$. We say they are essentially distinct if they are not comparable.
\end{definition}

In this paper, we will also consider the balls and $\de$-neighborhood on $A(k,n)$. Recall that we use $B_r(x)$ to denote the ball in $\R^n$ of radius $r$ centered at $x$. To distinguish the ambient space, we will use letter $Q$ to denote the balls in $A(k,n)$. For $V\in A(k,n)$, we use $Q_r(V)$ to denote the ball in $A(k,n)$ of radius $r$ centered at $V$. More precisely,
\[ Q_r(V_0):=\{V\in A(k,n): d(V,V_0)\le r  \}. \]
For a subset $X\subset A(k,n)$, we use the fancy letter $\cN$ to denote the neighborhood in $A(k,n)$:
\[ \cN_r(X):=\{V\in A(k,n): d(V,X)\le r\}. \]
Here, $d(V,X)=\inf_{V'\in X}d(V,V')$.

\bigskip

Next, we briefly recall how to define the Hausdorff dimension for subsets of a metric space. Let $(M,d)$ be a metric space. For $X\subset M$, we denote the $s$-dimensional Hausdorff measure of $X$ under the metric $d$ to be $\cH^s(X;d)$. 
We see that if $d'$ is another metric on $M$ such that $d(\cdot,\cdot)\sim d'(\cdot,\cdot)$, then $\cH^s(X;d)\sim \cH^s(X;d')$. It make sense to define the Hausdorff dimension of $X$ which is independent of the choice of comparable metrics:
\[ \dim X:= \sup\{s: \cH^s(X;d)>0\}. \]

Throughout the paper we will work with the metric space $M=A(k,n)$ and the metric $d$ defined as in \eqref{defdist}. For any $X\subset A(k,n)$, we will simply denote the $s$-dimensional Hausdorff measure of $X$ under the metric $d$ to be $\cH^s(X)$. Similarly, we use $\cH^s_\infty(X)$ to denote the Hausdorff content.

\subsection{Strategy of proof}
Our proof will follow the idea of Zahl \cite{zahl2022unions}. The main goal is to prove a Kakeya maximal estimate for $k$-planes, and then as a corollary we can prove Conjecture \ref{conj1}.  To handle with the $k$-planes rather than lines, we need to use the Brascamp-Lieb inequality, which can be viewed as a generalized version of the $k$-linear Kakeya estimate. We state our Kakeya maximal estimate for $k$-planes. This can be viewed as a counterpart of Theorem 1.2 in \cite{zahl2022unions}.

\begin{theorem}\label{kakeyathm}

Let $0<k\le d< n$ be integers and $\beta\in[0,1]$. Let $\cV=\{V\}$ be a subset of $A(k,n)$. Suppose that for all balls $Q_r\subset A(k,n)$ of radius $r\in [\de,1]$, we have
\begin{equation}\label{condition}
    \#\{ V\in\cV: V\in Q_r \}\le (r/\de)^{(k+1)(d-k)+\beta}. 
\end{equation} 
Then
\begin{equation}\label{kakeyaineq}
    \| \sum_{V\in\cV} 1_{V_\de}\|_{L^p(B^n(0,1))}\le C_\e \de^{-\frac{k(d-k)}{p'}-\e}\big(\sum_{V\in\cV}|V_\de|\big)^{1/p}.
\end{equation}
Here, 
\[p=\min\left\{ \frac{d-k+\beta+1}{d-k+\beta}, \frac{d+\beta}{d+\beta-(d-k+2)^{-1}},d-k+2\right\}\]
is a number strictly bigger than $1$.
\end{theorem}

\begin{remark}
    {\rm
For our range of $p$, the factor $\de^{-\frac{k(d-k)}{p'}}$ (up to the $\e$ loss) on the right hand side of \eqref{kakeyaineq} is sharp. Otherwise, it would give a better lower bound in \eqref{mainbound} which is impossible. The author is able to prove a sharp estimate for $p$ close to $1$. An interesting question would be to find the sharp factor $\de^{-C_p}$ for all $p\ge 1$.
    }
\end{remark}

The main idea in the proof of Theorem \ref{kakeyathm} is the Broad-Narrow method, which dates back to \cite{bourgain2010bounds}. Let us look at the left hand side of \eqref{kakeyaineq}. We first partition the integration domain $B^n(0,1)$ into $\de$-balls: $B^n(0,1)=\bigcup B_\de$. For each such $\de$-ball $B_\de$, consider all the $V_\de$ that intersect $B_\de$. We can morally view these $V_\de$ as a subset of $G(k,n)$. Here is the dichotomy: If there are not too ``many" $V_\de$ intersecting $B_\de$, then we say we are in the narrow case and it can be dealt with by the induction. If there are ``many" $V_\de$ intersecting $B_\de$, then most of these $V_\de$ are transverse to each other and we say we are in the broad case. The broad case is dealt with by the Brascamp-Lieb inequality. We will quantify the ``many" in the next section. 

In Section \ref{sec2}, we prove theorem \ref{kakeyathm}. In Section \ref{sec3}, we prove that Theorem \ref{kakeyathm} implies Conjecture \ref{conj1}.

\bigskip

 \noindent
 {\bf  Acknowledgement.}
 The author would like to thank Josh Zahl for helpful discussions.

\section{Proof of Theorem \ref{kakeyathm}}\label{sec2}

In this section we will use the Broad-Narrow method developed by Bourgain and Guth \cite{bourgain2010bounds}. To estimate the left hand side of \eqref{kakeyaineq}, we will decompose the integral into broad part and narrow part. The broad part will be handled by the Brascamp-Lied inequality, and the narrow part will be handled by the induction. 

\subsection{Brascamp-Lieb inequality}
The version of Brascamp-Lieb inequality we are going to use is due to Maldague. See Theorem 2 in \cite{maldague2022regularized}.

\begin{theorem}[Maldague]\label{Domthm}
Consider in $\R^n$.
    Let $W_j\subset G(n-k,n)$ for $j=1,\dots,J$. Fix $p\in[1,J]$. Define
    \begin{equation}\label{bl0}
        \BL(\{W_j\}_{j=1}^J,p):=\sup_{L\leq \R^n }(\dim L-\frac{p}{J}\sum_{j=1}^J\dim\pi_{W_j}(L)). 
    \end{equation} 
Here, $L\le \R^n$ means $L$ is a subspace of $\R^n$, and $\pi_{W_j}:\R^n\rightarrow W_j$ is the orthogonal projection. 

There exists $\nu>0$ depending on $\{W_j\}_{j=1}^J$, so that the following is true. For any $\cV_{j}=\{V_{j}\}\subset A(k,n)\ (j=1,\dots,J)$ being sets of $k$-planes, such that each $V_{j}\in \cV_{j}$ has its $k$ longest directions orthogonal to some $W\in Q_\nu(W_j)$, we have
\begin{equation}
    \int_{[-1,1]^n} \prod_{j=1}^J\bigg(\sum_{V_{j}\in\cV_{j}}1_{V_{j,\de}} \bigg)^{\frac{p}{J}}\lesssim \de^{n-\e}\de^{-\BL(\{W_j\}_{j=1}^J,p)}\prod_{j=1}^J\bigg(\#\cV_{j}\bigg)^{\frac{p}{J}}.
\end{equation}
\end{theorem}

By a compactness argument, we deduce the following result.

\begin{theorem}\label{BLthm}
Consider in $\R^n$.
    Let $\cW_j\subset G(n-k,n)$ be a compact set for $j=1,\dots,J$. Fix $p\in[1,J]$. 
Define
\begin{equation}\label{bl}
\BL(\{\cW_j\}_{j=1}^J,p):=\sup_{W_1\in\cW_1,\dots,W_J\in\cW_J }\BL(\{W_j\}_{j=1}^J,p). 
\end{equation} 
Then for any $\cV_{j}=\{V_{j}\}\subset A(k,n)\ (j=1,\dots,J)$ being sets of $k$-planes, such that each $V_{j}\in \cV_{j}$ has its $k$ longest directions orthogonal to some $W\in\cW_j$, we have
\begin{equation}\label{BL}
    \int_{[-1,1]^n} \prod_{j=1}^J\bigg(\sum_{V_{j}\in\cV_{j}}1_{V_{j,\de}} \bigg)^{\frac{p}{J}}\le C(\{\cW_j\},p,\e) \cdot \de^{n-\e}\de^{-\BL((\cW_j)_{j=1}^J,p)}\prod_{j=1}^J\bigg(\#\cV_{j}\bigg)^{\frac{p}{J}}.
\end{equation}
\end{theorem}

\subsection{Broad-narrow reduction}
We discuss the broad-narrow reduction in this subsection. There are some important submanifolds of $G(k,n)$ which resemble the great $(d-1)$-sphere $S^{d-1}$ in $S^{n-1}\simeq G(1,n)$. For any $d$-dimensional subspace $\Pi\subset \R^n$, let $G(k,\Pi)$ be the set of $k$-dimensional subspaces that are contained in $\Pi$. We see that $\dim G(k,\Pi)=k(d-k)$. 

Choose $\cD\subset G(k,n)$ to be a maximal $\de$-separated subset.
We first modify the directions of the $k$-planes in $\cV$ in Theorem \ref{kakeyathm} so that they have directions in $\cD$. For each $V\in\cV$, we can find another $V'$ so that $\dir(V')\in\cD$ and $V_\de\subset V'_{10\de}$.
We denote the new set of planes $\{V'\}$ obtained in this way by $\cV'$, then
\begin{equation}\label{rhs1}
    \|\sum_{V\in\cV}1_{V_\de} \|_{L^p(B^n(0,1) )}^p\le \| \sum_{V'\in\cV'}1_{V'_{10\de}} \|_{L^p(B^n(0,1) )}^p. 
\end{equation} 

We cover $B(0,1)$ by finitely overlapping $\de$-balls $\{B_\de\}$.
We bound the right hand side of \eqref{rhs1} as
\[ \| \sum_{V\in\cV'}1_{V_{10\de}} \|_{L^p(B^n(0,1) )}^p\le \sum_{B_\de} \| \sum_{V\in\cV'}1_{V_{10\de}} \|_{L^p(B_\de)}^p \lesssim\sum_{B_\de} \#\{V\in\cV': V_{10\de}\cap B_\de\neq \emptyset  \}^p \de^n. \]
Fix a $B_\de$, and let
\[ \cV'(B_\de):=\{V\in\cV': V_{10\de}\cap B_\de\neq \emptyset\}.\]
Our goal is to find an upper bound of $\#\cV'(B_\de)$.
Since the planes in $\cV'(B_\de)$ are $\de$-separated and intersect $B_{11\de}$, for any direction in $\cD$ there are $\lesssim 1$ planes in $\cV'(B_\de)$ that are in that direction. We can choose a subset $\cV(B_\de)\subset \cV'(B_\de)$, so that any two planes in $\cV(B_\de)$ have different directions and 
\[  \#\cV(B_\de)\gtrsim \#\cV'(B_\de).\]
Therefore, we bound the left hand side of \eqref{rhs1} by
\[ \lesssim\sum_{B_\de}\#\cV(B_\de)^p\de^n. \]
Since the planes in $\cV(B_\de)$ morally pass through $B_\de$ and have $\de$-separated directions, we can view $\cV(B_\de)$ as a $\de$-separated subset of $G(k,n)$.


Fix a $B_\de$, there is a dichotomy to estimate $\#\cV(B_\de)$. We talk about the heuristic. The first scenario is when there exists a $d$-dimensional subspace $\Pi\in G(d,n)$ so that most of the elements $V\in\cV(B_\de)$ lie close to $G(k,\Pi)$ (if we view $V$ as a point in $G(k,n)$), or equivalently most of the elements $V\in\cV(B_\de)$ lie in a small neighborhood of $\Pi(\subset \R^n)$ (if we view $V$ as a $k$-plane in $\R^n$). The second scenario is that we can find several ``quite transverse" patches $\cU_1,\dots, \cU_J\subset G(k,n)$ so that $\#(\cU_j\cap \cV(B_\de))\gtrsim \#\cV(B_\de)$ for $j=1,\dots,J$. The transverse condition on the patches allows us to use Theorem \ref{BLthm}. 
We call the first scenario narrow case, and the second scenario broad case.
We make our heuristic rigorous as follows.

Fix a large number $K\gg1$ which is to be determined later. We cover $G(k,n)$ by $\cQ_{K^{-n}}=\{Q_{K^{-n}}\}$ which are balls of radius $K^{-n}$ with bounded overlaps: 
\[ G(k,n)=\bigcup Q_{K^{-n}}. \]
Note that $\#\cQ_{K^{-n}}\lesssim (K^n)^{\dim G(k,n)}\le K^{n^3}$. We define the significant balls
\[ \cQ'=\{ Q_{K^{-n}}: \#(Q_{K^{-n}}\cap \cV(B_\de))\ge \frac{1}{K^{n^4}}\#\cV(B_\de) \}. \]
We have 
\[ \#\cV(B_\de)\le 2 \sum_{Q_{K^{-n}}\in \cQ'}\#(Q_{K^{-n}}\cap \cV(B_\de)). \]
Next by pigeonhole principle, there is a subset of $\cQ'$ denoted by $\cQ$, such that the balls in $\cQ$ are $100K^{-n}$-separated, $\#(Q_{K^{-n}}\cap \cV(B_\de))$ are all comparable for $Q_{K^{-n}}\in\cQ$, and
\[ \#\cV(B_\de)\lesssim \log K \sum_{Q_{K^{-n}}\in\cQ}\#(Q_{K^{-n}}\cap \cV(B_\de)). \]

There is the dichotomy: 

\medskip 

\fbox{The first scenario}: We say it is in the narrow case when there exists $\Pi\in G(d,n)$ so that half of balls in $\cQ$ is contained in $\cN_{CK^{-1}}(G(k,\Pi))$. Here, we view $G(k,\Pi)$ as a subset of $G(k,n)$ and $\cN_{C_1K^{-1}}(G(k,\Pi))$ is the $C_1K^{-1}$-neighborhood of $G(k,\Pi)$ in $G(k,n)$. $C_1$ is a large constant to be determined later.

\medskip
\fbox{The second scenario}: We say it is in the broad case when there exists 
\[\cU_k,\dots,\cU_{d+1}\in \cQ,\]
so that there does not exists a $\Pi\in G(d,n)$ such that all of these $\cU_j$  intersect $\cN_{100K^{-n}}(G(k,\Pi))$. In other words, for any $\Pi\in G(d,n)$, there exists a $\cU_j$ that does not intersect $\cN_{100K^{-n}}(G(k,\Pi))$. We remark that the subscript of $\cU_j$ ranges from $k$ to $d+1$ but not from $1$ to $d-k+2$, just to make the discussion below more convenient.

\bigskip

We want to show that if we are not in the first scenario, then we are in the second scenario. 
We let $\cG$ be the set of centers of the balls in $\cQ$, i.e., 
\[ \cQ=\{Q_{K^{-n}}(U):U\in\cG\}. \]
We just need to show there exists $U_k,\dots,U_{d+1}\in\cG$ so that there does not exists a $\Pi\in G(d,n)$ such that all these $U_j$ are contained in $\cN_{200K^{-n}}(G(k,\Pi))$.

If $U_k,\dots,U_{d+1}\in \cN_{200K^{-n}}(G(k,\Pi))$, then by Lemma \ref{comparablelem}, 
\begin{equation}\label{nothappen}
    U_k,\dots,U_{d+1}\subset N_{CK^{-n}}(\Pi)
\end{equation}
for some large constant $C$ depending on $k,n$.
To show the existence of $U_k,\dots,U_{d+1}$ that do not satisfy \eqref{nothappen}, we will inductively find $U_k,\dots,U_{d+1}$ and vectors $v_1,\dots,v_{d+1}$ in $\bigcup_{j=k}^{d+1}V_j$ such that $v_1\dots,v_{d+1}$ cannot all lie in $N_{CK^{-n}}(\Pi)$.

We observe the following fact: If the unit vectors $v_1,\dots,v_{d+1}$ lie in $CK^{-n}$-neighborhood of a $d$-dimensional plane $\Pi$, then the volume of the parallelepiped spanned by $\{v_1,\dots,v_{d+1}\}$ has the upper bound: 
\[\tx{Vol}_{d+1}(v_1,\dots,v_{d+1})\lesssim  K^{-n}.\] The proof is straightforward. Let $V^{d+1}=\tx{span}\{v_1,\dots,v_{d+1}\}$, then parallelepiped spanned by $\{v_1,\dots,v_{d+1}\}$ is contained in $N_{K^{-n}}(\Pi)\cap V^{d+1}\cap B^n(0,1)$. We get 
\begin{equation}
    \tx{Vol}_{d+1}(v_1,\dots,v_{d+1})\le \tx{Vol}_{d+1}\bigg( N_{K^{-n}}(\Pi)\cap B^n(0,1)\cap V^{d+1} \bigg).
\end{equation}
Noting that $N_{K^{-n}}(\Pi)\cap B^n(0,1)$ is morally a rectangle of dimensions \[\underbrace{1\times \dots\times 1}_{d \tx{~times} }\times \underbrace{K^{-n}\times \dots \times K^{-n}}_{(n-d)\tx{~times} }.\]
Its intersection with a $d+1$ dimensional plane has volume $\lesssim K^{-n}$.

Therefore, we just need to find $U_k,\dots,U_{d+1}\in\cG$ and vectors $v_1,\dots,v_{d+1}$ that lie in $\bigcup_{j=k}^{d+1} V_j$ such that
\[ \tx{Vol}_{d+1}(v_1,\dots,v_{d+1})\ge C' K^{-n}, \]
for some large constant $C'$. This will implies that $U_k,\dots,U_{d+1}$ cannot all lie in $\cN_{200K^{-n}}(G(k,\Pi))$.

This is done in the following way. 
We first choose a $U_k\in\cG$ and choose $v_1,\dots,v_k$ to be the orthonormal basis of $U_k$.
Next, we will inductively construct $U_{j}$ for $k<j\le d+1$ and a unit vector $v_j\in U_j$ so that the $j$-dimensional parallelepided spanned by $v_1,\dots,v_j$ has volume bigger than $\wt C^jK^{-j}$ for some large constant $\wt C$. Suppose we have constructed $\cU_1,\dots,\cU_{j-1}$ and $v_1,\dots,v_{j-1}$ with $\textup{Vol}_{j-1}(v_1,\dots,v_{j-1})\ge \wt C^{j-1}K^{j-1}$. Denote the $(j-1)$-dimensional plane $\textup{span}\{v_1,\dots,v_{j-1}\}$ by $\Pi_{j-1}$. For any $U\in\cG$, consider its relative position with $\Pi_{j-1}$. It suffices to find a unit vector $v\in U$ that is not contained in the $\wt C K^{-1}$-neighborhood of $\Pi_{j-1}$. The key observation is that: If any unit vector $v\in U$ is contained in the $\wt CK^{-1}$-neighborhood of $\Pi_{j-1}$, then $U$, as an element in $G(k,n)$, is contained in $\cN_{C_1 K^{-1}}(G(k,\Pi_{j-1}))$. (Here, $C_1$ is the constant in \fbox{The first scenario} which is determined after $\wt C$.) The proof is as follows. If any unit vector $v\in U$ is contained in the $\wt C K^{-1}$-neighborhood of $\Pi_{j-1}$, then if we view $U$ and $\Pi_{j-1}$ as subsets of $\R^n$, we see that $B(0,1)\cap U$ is contained in the $\wt C K^{-1}$-neighborhood of $\Pi_{j-1}$. If we denote the orthogonal projection of $U$ to $\Pi_{j-1}$ by $U'$, then $B(0,1)\cap U$ is contained in the $\wt C K^{-1}$-neighborhood of $U'$. By Lemma \ref{comparablelem}, $d(U,U')\lesssim C\wt C K^{-1}$. Since $U'\in G(k,\Pi_{j-1})$, we have $d(U,G(k,\Pi_{j-1}))\le C_1 K^{-1}$, where $C_1\sim C \wt C$ is a large constant.  In other words, all the $\cU\in\cQ$ are contained in $\cN_{C_1K^{-1}}(G(k,\Pi_{j-1}))$, meaning that we are in the first scenario.

\bigskip

Let us summarize the result.
Let $\cQ_{K^{-1}}=\{Q_{K^{-1}}\}$ be a covering of $G(k,n)$ by $K^{-1}$-balls.
If we are in the first scenario (broad case), then there exists $\Pi\in G(d,n)$ such that
\[ \#\cV(B_\de)\lesssim \log K \sum_{Q_{K^{-1}}\subset \cN_{C_1K^{-1}}(G(k,\Pi))}\#(Q_{K^{-1}}\cap \cV(B_\de)). \]
Since $G(k,\Pi)$ is a $k(d-k)$-dimensional submanifold of $G(k,n)$, we have 
\[\#\{Q_{K^{-1}}\subset \cN_{C_1K^{-1}}(G(k,\Pi))\}\lesssim K^{k(d-k)}. \]
Therefore, we have

\begin{equation}
    \#\cV(B_\de)^p\lesssim K^{k(d-k)(p-1)} (\log K)^p \sum_{Q_{K^{-1}}\in \cQ_{K^{-1}}}\#(Q_{K^{-1}}\cap \cV(B_\de))^p.
\end{equation}

We discuss the bound we have in the second scenario. Let $\cQ_{K^{-n}}$ be a 
set of $K^{-n}$-cubes that form a partition of $G(k,n)$. Define the following transverse $(d-k+2)$-tuples
\begin{align}
    \textup{Trans}^{d-k+2}:=\{ (\cU_k,\dots,\cU_{d+1})\in \cQ_{K^{-n}}^{d-k+2}:\cU_k,\dots,\cU_{d+1} \textup{~do~not~simultaneously}\\
   \nonumber \textup{~interset~}\cN_{100K^{-n}}(G(k,\Pi)) \textup{~for~any~} \Pi\in G(d,n)  \}.
\end{align}

Then if we are in the second scenario, we have
\begin{equation}
    \#\cV(B_\de)\lesssim K^{O(n^{10})} \sum_{(\cU_k,\dots,\cU_{d+1})\in\textup{Trans}^{d-k+2}}\prod_{j=k}^{d+1}(\#\cU_j\cap\cV(B_\de))^{\frac{1}{d-k+2}}.
    \end{equation}

Combining the broad case and narrow case, we obtain the following estimate.
\begin{align}
    \#\cV(B_\de)^p\lesssim K^{k(d-k)(p-1)}(\log K)^p\sum_{Q_{K^{-1}}\in\cQ_{K^{-1}}}\#(Q_{K^{-1}}\cap\cV(B_\de))^p\\
    \label{huge}+K^{pK^{n^3}} \sum_{(\cU_k,\dots,\cU_{d+1})\in\textup{Trans}^{d-k+2}}\prod_{j=k}^{d+1}(\#\cU_j\cap\cV(B_\de))^{\frac{p}{d-k+2}}.
\end{align}

\begin{remark}
    {\rm
     \eqref{huge} is obtained by using triangle inequality which results in a huge constant $K^{pK^{n^3}}$ depending on $K$. This constant comes from $\#\{(\cU_k,\dots,\cU_{d+1})\in\textup{Trans}^{d-k+2}\}$. However, it is allowable since it does not depend on $\de$.
    }
\end{remark}

Summing over $B_\de$ and noting that each $V_\de$ is roughly constant on any ball of radius $\de$, we obtain

\begin{align}
    \|\sum_{V\in\cV}1_{V_\de}\|_p^p\lesssim K^{k(d-k)(p-1)}(\log K)^p\sum_{Q_{K^{-1}}\in\cQ_{K^{-1}}}\| \sum_{V\in\cV, \dir(V)\in Q_{K^{-1}}}1_{V_{20\de}} \|_p^p\\
    +C_K\sum_{(\cU_k,\dots,\cU_{d+1})\in\textup{Trans}^{d-k+2}}\int_{[-1,1]^n}\prod_{j=k}^{d+1}(\sum_{V\in\cV, \dir(V)\in \cU_j}1_{V_{20\de}})^{\frac{p}{d-k+2}}.
\end{align}

To ease the notation, we may replace $V_{20\de}$ on the right hand side by $V_\de$. This may cost a lose of some constant factor, which is allowable. Therefore, we can assume we have

\begin{align}
    \label{1term}\|\sum_{V\in\cV}1_{V_\de}\|_p^p\lesssim K^{k(d-k)(p-1)}(\log K)^p\sum_{Q_{K^{-1}}\in\cQ_{K^{-1}}}\| \sum_{V\in\cV, \dir(V)\in Q_{K^{-1}}}1_{V_\de} \|_p^p\\
    \label{2term}+C_K\sum_{(\cU_k,\dots,\cU_{d+1})\in\textup{Trans}^{d-k+2}}\int_{[-1,1]^n}\prod_{j=k}^{d+1}(\sum_{V\in\cV, \dir(V)\in \cU_j}1_{V_\de})^{\frac{p}{d-k+2}}=:\textup{I}+\textup{II}.
\end{align}

\bigskip

Next, we will estimate I and II separately. We first look at II. Let $J=d-k+2$. For each $\cU_j$ ($k\le j\le d+1$), let $\cW_j=\{W^\perp:W\in \cU_j\}$. We see that $\cW_j\subset G(n-k,n)$. For simplicity, we denote the Brascamp-Lieb  constant (see \eqref{bl} and \eqref{bl0}) by
\begin{equation}\label{recallBL}
    \BL:=\BL((\cW_j)_{j=1}^J,p). 
\end{equation} 
Noting $p\le d-k+2=J$, by Theorem \ref{BLthm}, we have
\begin{equation}
    \int_{[-1,1]^n}\prod_{j=k}^{d+1}(\sum_{V\in\cV, \dir(V)\in \cU_j}1_{V_\de})^{\frac{p}{d-k+2}}\lesssim \de^{n-\e^2}\de^{-\BL}\prod_{j=1}^J\bigg(\#\cU_{j}\cap \cV\bigg)^{\frac{p}{J}}\le\de^{n-\e^2-\BL}(\#\cV)^p.
\end{equation}
Here, the implicit constant depends on $\e$ and $\{\cW_j\}$. $\{\cW_j\}$ further depends on the covering $\cQ_{K^{-n}}$. Since this covering is at the scale $K^{-n}$, we can simply denote this implicit constant by $C_{K,\e}$. 
Also, we choose $\de^{-\e^2}$ instead of $\de^{-\e}$ is to close the induction later.

Therefore,
\[ \tx{II}\le C_{K,\e} \de^{n-\e^2-\BL}(\#\cV)^p.\]
Here, $C_K$ may be another constant different from that in \eqref{2term}.
We claim that 
\begin{equation}\label{clm}
    \de^{n-\BL}(\#\cV)^p\lesssim \de^{-k(d-k)(p-1)}\sum_{V\in\cV}|V_\de|. 
\end{equation}
To prove this inequality, we first note that $|V_\de|\sim \de^{n-k}$ and by \eqref{condition} that $\#\cV\le \de^{-(k+1)(d-k)+\beta}$.
Therefore, \eqref{clm} is equivalent to
\[ \de^{n-\BL}\de^{-\big((k+1)(d-k)+\beta\big)(p-1)}\lesssim \de^{-k(d-k)(p-1)}\de^{n-k}. \]
It suffices to prove
\begin{equation}\label{boundBL}
    \BL\le d+\beta-(d-k+\beta)p. 
\end{equation} 
Recall the definition of $\BL$ in \eqref{recallBL}. We may assume
\[ \BL=\sup_{L\le \R^n}\bigg(\dim L-\frac{p}{J}\sum_{j=k}^{d+1}\dim \pi_{W_j}(L)\bigg), \]
where $W_j\in \cW_j$. We remind the reader that $\{W_j\}_{j=k}^{d+1}$ satisfy an important property: any $\Pi\in G(d,n)$ does not contain all the $W_j^\perp$. We will also use the following result.
\begin{lemma}\label{easylem}
    Let $W,L\le \R^n$ be two subspaces and $\dim L+\dim W\ge n$. Then,
     \[\dim\pi_{W}(L)\ge \dim W+\dim L-n.\] 
     The equality holds if and only if $L\supset W^\perp$. If $L\not\supset W^\perp$, then \[\dim\pi_{W}(L)\ge \dim W+\dim L-n+1.\]
\end{lemma}
\begin{proof}
    We just look at the linear map $\pi_W|_L: L\rightarrow \pi_W(L)$. We have $\dim(\pi_W(L))=\dim(L)-\dim\textup{Ker}(\pi_W|_L)$. We note that $\dim\textup{Ker}(\pi_W|_L)\le \dim\textup{Ker}(\pi_W)=n-\dim W$, and $\dim\textup{Ker}(\pi_W|_L)= \dim\textup{Ker}(\pi_W)$ if and only if $L\supset W^\perp$.
\end{proof}

We return to the proof of \eqref{boundBL}.
We denote $l=\dim L$ and consider three cases.
\medskip

\begin{itemize}
    \item $l\ge d+1$: $\dim L-\frac{p}{J}\sum_{j=k}^{d+1}\dim \pi_{W_j}(L)\le l-p(n-k+l-n)=l-p(l-k)\le d+\beta-(d-k+\beta)p$, since $p>1$.
    \medskip
    \item $k\le l\le d$: Since $L$ cannot contain all the $W_j^\perp$, by Lemma \ref{easylem} we have \[\sum_{j=k}^{d+1}\dim\pi_{W_j}(L)\ge J(l-k)+1 .\]
    Therefore,
    $\dim L-\frac{p}{J}\sum_{j=k}^{d+1}\dim \pi_{W_j}(L)\le l-p(l-k)-\frac{p}{J}\le d+\beta-(d-k+\beta)p$. The last inequality is equivalent to $(d+\beta-l-J^{-1})p\le d+\beta-1$. This is true if either $d+\beta-l-J^{-1}\le 0$, or $p\le \frac{d+\beta-l}{d+\beta-l-J^{-1}}$.
    \medskip
    \item $l\le k-1$: We simply use $\dim L-\frac{p}{J}\sum_{j=k}^{d+1}\dim \pi_{W_j}(L)\le \dim L\le k-1\le d+\beta-(d-k+\beta)p$, where the last inequality is equivalent to
    \[ p\le \frac{d+1-k+\beta}{d-k+\beta}.\]  
\end{itemize}
We have finished the proof of claim \eqref{clm}. We obtain that 
\begin{equation}\label{estII}
    \tx{II}\le C_{K,\e}\de^{-\e^2-k(d-k)(p-1)}\sum_{V\in\cV}|V_\de|.
\end{equation}

\bigskip

We turn to the term $\tx{I}$ which is the right hand side of \eqref{1term}. For a fixed $Q_{K^{-1}}\in \cQ_{K^{-1}}$, we choose $V^\circ\in G(k,n)$ to be the center of $Q_{K^{-1}}$. Let $V^\circ_{K^{-1}}$ denote $N_{K^{-1}}(V^\circ)\cap B^n(0,1)$ which is a $K^{-1}$-slab in $\R^n$, centered at the origin. For any $V\in Q_{K^{-1}}$, we have $\dir(V)\cap B^n(0,1)\subset C V_{K^{-1}}^\circ$. Therefore, $V_\de$ is contained in a translation of $C V^\circ_{K^{-1}}$.

To estimate $\| \sum_{V\in\cV, \dir(V)\in Q_{K^{-1}}}1_{V_\de} \|_{L^p(B^n(0,1))}^p$, we will first partition the integration domain. We cover $B^n(0,1)$ by finitely overlapping slabs denoted by $\cT=\{\tau\}$. Here each $\tau$ is a translation of $ C V^\circ_{K^{-1}}$, and $\#\cT\sim K^{n-k}$.
Also, $\cT$ is chosen so that for any $V\in Q_{K^{-1}}$, $V_\de$ is contained in some $\tau$. If $V_\de$ is contained in $\tau$, than we associate $V_\de$ to $\tau$ and denote by $V_\de\subset \tau$. If there are multiple choices, we just associate $V_\de$ to one $\tau$. We have
\begin{equation}\label{ineq0}
    \| \sum_{V\in\cV, \dir(V)\in Q_{K^{-1}}}1_{V_\de} \|_{L^p(B^n(0,1))}^p=\| \sum_{\tau}\sum_{V\in\cV, \dir(V)\in Q_{K^{-1}},V_\de\subset \tau}1_{V_\de} \|_{L^p(B^n(0,1))}^p. 
\end{equation} 
We note that $\supp \left(\sum_{V\in\cV, \dir(V)\in Q_{K^{-1}},V_\de\subset \tau}1_{V_\de}\right)\subset \tau$ and $\cT=\{\tau\}$ are finitely overlapping, so the equation above is bounded by
\begin{equation}\label{ineq1}
    \lesssim \sum_\tau \|\sum_{V\in\cV, \dir(V)\in Q_{K^{-1}},V_\de\subset \tau}1_{V_\de}\|_{L^p(\tau)}^p. 
\end{equation} 
Next, we will do the rescaling for $\tau$. We first translate $\tau$ to the origin, and do the $K$-dilation
in the short directions of $\tau$ so that $\tau$ roughly becomes $B^n(0,1)$. We denote this rescaling by $\Phi_\tau$. After the rescaling, we see that every $V_\de\subset \tau$ becomes a $\de K\times \dots\times \de K\times 1\times \dots\times 1$-slab contained in $B^n(0,1)$. We can just assume it is $V'_{\de K}$, where $V'=\Phi_\tau(V)$. We denote $\cV'(\tau):=\{\Phi_\tau V: V_\de\subset \tau\}$.
We have
\begin{equation}\label{ineq2}
    \|\sum_{V\in\cV, \dir(V)\in Q_{K^{-1}},V_\de\subset \tau}1_{V_\de}\|_{L^p(\tau)}^p\sim K^{-(n-k)}\|\sum_{V'\in\cV'(\tau)}1_{V'_{\de K}}\|_{L^p(B^n(0,1))}^p 
\end{equation} 

Next, we check the spacing condition \eqref{condition} for $\cV'(\tau)$. We want to prove that for any ball $Q_r\subset A(k,n)$ of radius $r$ ($\de K\le r\le 1$), we have
\[ \#\{V'\in \cV'(\tau): V'\in Q_r\}\lesssim (r/\de K)^{(k+1)(d-k)+\beta}. \]
Suppose $W'\in A(k,n)$ is the center of $Q_r$. Then, $V'\in Q_r$ implies $V'_{\de K}\subset W'_{Cr}$ by Lemma \ref{comparablelem}. Let $W=\Phi_{\tau}^{-1}W'$.
If we rescale back, i.e., we apply $\Phi_\tau^{-1}$, then $W'_{Cr}$ becomes $W_{CrK^{-1}}$ which is a $CrK^{-1}$-slab. The condition $V'_{\de K}\subset W'_{Cr}$ becomes $V_{\de}\subset W_{CrK^{-1}}$. Again by Lemma \ref{comparablelem}, this condition implies $V\in Q_{C^2r K^{-1}}$ where $Q_{C^2 rK^{-1}}\subset A(k,n)$ is a ball centered at $W$.
In summary, We proved that
\[ \#\{V'\in \cV'(\tau): V'\in Q_r\}\le \#\{V\in \cV: V\in Q_{C^2 r K^{-1}}\}.  \]
By \eqref{condition}, it is $\lesssim (r/\de K)^{(k+1)(d-k)+\beta} $. 

\bigskip

\subsection{Induction on scales}
We are ready to prove \eqref{kakeyaineq}. We will induct on $\de$. Of course, \eqref{kakeyaineq} is true for small $\de$ if we choose $C_\e$ in \eqref{kakeyaineq} sufficiently large. Suppose \eqref{kakeyaineq} is true for numbers $\ge 2\de$. We will use induction hypothesis for the scale $\de K$ to bound \eqref{ineq2}. However, $\cV'(\tau)$ satisfies
\begin{equation}\label{condi}
    \#\{V'\in \cV'(\tau): V'\in Q_r\}\le C\cdot (r/\de K)^{(k+1)(d-k)+\beta}, 
\end{equation} 
where there is an additional constant $C$ compared with \eqref{condition}. We hope to partition $\cV'(\tau)$ into subsets so that each subset satisfies the inequality above with $C=1$. This is done by the following lemma.

\begin{lemma}\label{goodlem}
    Let $A\subset [0,1]^m$ be a finite set that satisfies
    \[ \#(A\cap Q_r)\le M (r/\de)^s, \]
    for any $Q_r$ being a ball of radius $r$ in $[0,1]^m$ for $\de\le r\le 1$. Then we can partition $A$ into $\lesssim 10^{5m}M^2$ subsets $A=\sqcup A'$ such that each $A'$ satisfies
    \[ \#(A'\cap Q_r)\le (r/\de)^s, \]
    for any $Q_r$ being a ball of radius $r$ for $\de\le r\le 1$.
\end{lemma}

\begin{proof}
For convenience, we assume $\de$ is a dyadic number.
     For any dyadic number $r\in [\de,1]$, let $\cD_r=\{D_r\}$ be lattice cubes of length $r$ that form a partition of $[0,1]^m$. Note that every ball of radius $r_0$ is contained in $4^m$ dyadic cubes of length $r$, where $r\in[r_0/2, r_0]$ is dyadic. We just need to find a partition $A=\sqcup A''$ such that for each $A''$,
    \[ \#(A''\cap D_r)\le \max\{1, 10^{-m}(r/\de)^s\}. \]
We first partition $A$ into $\lesssim 10^{m}M$ subsets $\sqcup A'$ so that $\#(A'\cap D_r)\le 1$ for $r\le 10^{\frac{m}{s}}\de$. It remains to partition each $A'$ further into subsets $A'=\sqcup A''$ so that
\[ \#(A''\cap D_r)\le 10^{-m}(r/\de)^s, \tx{~for~}r>10^{\frac{m}{s}}\de. \]

Fix a $A'$ in the first stage of partition.    
    We will construct a sequence of graphs $G_r=(V_r,E_r)$ indexed by the dyadic numbers $r\in [\de,1]$, so that their vertex sets are $V_r=A'$ and their edge sets satisfy $E_r\subset E_{r'}$ if $r<r'$. Also, each $G_r$ is a disjoint union of cliques of order $\le 10^{3m}M$ and the number of cliques of $G_r$ in any dyadic cube $D_r$ is $\le 10^{-m}(r/\de)^s$. Here, a clique is a complete graph, and the order of a clique is the number of points in the clique.

    For $r\le 10^{\frac{m}{s}}\de$, we just let $E_r$ be empty. Suppose we have constructed for $r/2$, so $E_{r/2}$ is a disjoint union of cliques with order $\le 10^{3m}M$. We are going to construct $E_r$. Fix a $D_r$. Consider the cliques of $E_{r/2}$ that lie in $D_r$. If there are two cliques, both of which have order $\le \frac{1}{2}10^{3m}M$, we merge them to get a clique of order $\le 10^{3m}M$. We continue the process until there is at most one clique whose order is $\le \frac{1}{2}10^{3m}M$. We do this argument for each $D_r$, and we obtain $E_r$. We need to check that for each $D_r$, there are $\le 10^{-m}(r/\de)^s$ cliques. First of all,
    $\#(A'\cap D_r)\le M(r/\de)^s$. On the other hand, $\#(A'\cap D_r)\ge (\#(\tx{cliques~in~}D_r)-1)\cdot\frac{1}{2}10^{3m}M$.
    Therefore, $\#(\tx{cliques~in~}D_r)\le 10^{-m}(r/\de)^s$. We repeat the algorithm and finally obtain $G_1=(A',E_1)$.
    
     Since each clique in $G_1$ has order $\le 10^{3m}M$, we can partition $A'$ into $\le 10^{3m}M$ sets $A'=\sqcup A''$, so that for each $A''$ there does not exists two points in $A''$ that lie in the same clique.  Now, we have
    \[ \#(A''\cap D_r)\le \#(\tx{cliques~in~}D_r)\le 10^{-m}(r/\de)^s. \]
This finishes the proof.
    
\end{proof}

By the lemma, we can partition $\cV'(\tau)=\bigsqcup_j \cV'_{j}(\tau)$ into $O(1)$ subsets, so that \eqref{condi} holds for $\cV'_{j}(\tau)$ with $C=1$. Applying the induction hypothesis \eqref{kakeyaineq} and triangle inequality, we obtain
\[ \|\sum_{V'\in \cV'(\tau)}1_{V'_{\de K}}\|_{L^p(B^n(0,1))}^p\le C C_\e (\de K)^{-k(d-k)(p-1)-\e p}(\sum_{V'\in \cV'(\tau)} |V'_{\de K}|). \]
Summing over $\tau\in \cT$, we can bound the left hand side of \eqref{ineq0} as
\begin{align*}
     &\|\sum_{V\in \cV, \dir(V)\in Q_{K^{-1}}}1_{V_{\de}}\|_{L^p(B^n(0,1))}^p\\
     \le &CK^{-(n-k)} C_\e (\de K)^{-k(d-k)(p-1)-\e p}\sum_{V'\in \cV'(\tau)} |V'_{\de K}|\\
     \lesssim &C C_\e (\de K)^{-k(d-k)(p-1)-\e p}\sum_{V\in \cV, \dir(V)\in Q_{K^{-1}}} |V_{\de}|
\end{align*}
(The constant $C$ changes from line to line, but is independent of $K$ and $\de$.)

Summing over $Q_{K^{-1}}\in \cQ_{K^{-1}}$, we obtain
\[ \tx{I}\le C (\log K)^pK^{-\e p}  C_\e \de^{-(k(d-k)(p-1)-\e p}\sum_{V\in\cV}|V_\de|. \]
Combining with \eqref{estII}, we can bound the left hand side of \eqref{1term} by
\begin{align*}
\|\sum_{V\in\cV}1_{V_\de}\|_p^p\lesssim C (\log K)^pK^{-\e p}  C_\e \de^{-k(d-k)(p-1)-\e p}\sum_{V\in\cV}|V_\de|\\
+C_{K,\e}\de^{-\e^2-k(d-k)(p-1)}\sum_{V\in\cV}|V_\de|.
\end{align*} 
To close the induction, it suffices to make
\begin{equation}\label{1}
    C (\log K)^pK^{-\e p} \le 1/2 
\end{equation} 
and
\begin{equation}\label{2}
    C_{K,\e}\de^{-\e^2}\le \frac12 \de^{-\e p}C_\e.
\end{equation}
We first choose $K$ sufficiently large so that \eqref{1} holds. Then we choose $C_\e$ large enough so that \eqref{2} holds.

\section{Proof of Conjecture \ref{conj1}}\label{sec3}

It is shown in \cite[Section 3]{zahl2022unions} that Theorem \ref{kakeyathm} implies Conjecture \ref{conj1} in the case $k=1$. Actually, for all the $k$, the implication is the same. Here, we give a sketch of proof on why Theorem \ref{kakeyathm} implies Conjecture \ref{conj1}.
We skip the technical details on the measurability and focus on how the numerology works for general $k$.

\begin{proof}[A sketch of proof: Theorem \ref{kakeyathm} implies Conjecture \ref{conj1}]
We just need to consider the case when $\beta\in (0,1]$. If $\beta>1$, we just replace $\beta$ by $1$ and throw away some $k$-planes in $\cV$ to make $\dim (\cV)=(k+1)(d-k)+1$. 
If $\beta=0$, we just replace $\beta$ by $1$ and $d$ by $d-1$. Next, note that it suffices to prove a variant of \eqref{mainbound}, where the right hand side is replaced by $d+\beta-\e$ for arbitrary $\e>0$.

Fix such an $\e$. Let $X=\bigcup_{V\in\cV}V$. We will first localize $X$ inside $B^n(0,1)$.
Define
\[ \cV^N:=\{V\in\cV: \dist(0,V)< N/2  \}. \]
Note that $\cV^N$ ($N=1,2,\cdots$) is an ascending sequence that converges to $\cV$. Therefore, there exists $N$ such that
\[ \dim (\cV^N)\ge \dim (\cV)-\e/2. \]
Note that after the rescaling $x\mapsto x/N$, the Hausdorff dimensions of $\cV,\cV^N$ and  $X$ do not change.
We use $\cV_1$ to denote the $\cV^N$ after the rescaling. By abuse of notation, we still keep the notation for $\cV, X$ after the rescaling. We have
\[ \cV_1=\{V\in\cV: \dist(0,V)< 1/2\}\subset A(k,n), \]
and $\dim(\cV_1)\ge \dim(\cV)-\e/2$. By Frostman's lemma, there is a probability measure $\mu$ supported on $\cV_1$, such that for any $Q_r\subset A(k,n)$ being a ball of radius $r$ we have
\begin{equation}\label{frost}
    \mu(Q_r)\le C_0 r^{(k+1)(d-k)+\beta-\e}. 
\end{equation} 
We also define the localized version of $X$:
\[ X_1:=\bigcup_{V\in\cV_1}V\cap B^n(0,1). \]
We remark that for any $V\in\cV_1$,
\[ |V\cap B^n(0,1)|\gtrsim 1. \]
Our goal is to show 
\begin{equation}
    \dim(X_1)\ge d+\beta-O(\e).
\end{equation}

Let $k_0$ be a large integer that will be chosen later.
Cover $X_1$ by a union $\bigcup_{k=k_0}^\infty \bigcup_{B\in\cB_k}B$, where $\cB_k$ is a collection of balls of radius $2^{-k}$, and 
\begin{equation}\label{lowerbound}
    \sum_k 2^{-k(\dim X_1+\e)}\#\cB_k=:C_1<\infty.
\end{equation}

For each $V\in\cV_1$, there is an index $k=k(V)\ge k_0$, so that
\begin{equation}\label{lowerbound2}
    |V\cap \bigcup_{B\in\cB_k}B|\gtrsim \frac{1}{k^2}. 
\end{equation} 
For each index $k$, define 
\[\cV^{(k)}=\{V\in\cV_1:k(V)=k\}.\] 
Then 
\[\sum_{k\ge k_0}\mu(\cV^{(k)})=1, \]
so there exists an index $k_1\ge k_0$ with $\mu(\cV^{(k_1)})\ge \frac{1}{k_1^2}$. From now on, we just define $\de=2^{-k_1}$ and $E_\de=\bigcup_{B\in\cB_{k_1}}B.$

From \eqref{lowerbound}, we see that

\begin{equation}\label{lowerbd}
    |E_\de|=\#\cB_{k_1}2^{-k_1}\le C_1 \de^{n-\e-\dim X_1}.
\end{equation}

Noting that $\mu$ is a $\big((k+1)(d-k)-\e\big)$-dimensional Frostman measure (see \eqref{frost}), by \cite[Lemma 8]{dote2022exceptional}, there exists a subset $\cV_2\subset \cV^{(k_1)}$ that satisfies $\# \cV_2\gtrsim \frac{1}{k_1^2}\de^{\e-(k+1)(d-k)-\beta}$, and
\[ \#\{V\in\cV_2: V\in Q_r\}\lesssim \de^{-\e}(r/\de)^{(k+1)(d-k)+\beta} \]
for any $Q_r\subset A(k,n)$ being a ball of radius $r$ ($\de\le r\le 1$).

In what follows, we will write $A\lessapprox B$ to mean $A\lesssim |\log\de|^{O(1)} B$.

By Lemma \ref{goodlem}, we can find $\cV_3\subset \cV_2$ so that $\# \cV_3\gtrsim \de^{O(\e)-\big((k+1)(d-k)+\beta\big)}$, and $\cV_3$ satisfies
\[ \#\{V\in\cV_3: V\in Q_r\}\le (r/\de)^{(k+1)(d-k)+\beta}. \]
From \eqref{lowerbound2}, we see that for any $V\in\cV_3$, 
\begin{equation}\label{lowerbd2}
    |V_\de\cap E_\de|\gtrapprox 1\sim |V_\de|.
\end{equation} 

Now, we can apply Theorem \ref{kakeyathm} to $\cV_3$. 
By \eqref{kakeyaineq} and H\"older's inequality, we have
\[ \big(\int_{E_\de} \sum_{V\in\cV_3} 1_{V_\de}\big)\big| 
E_\de\big|^{-1/p'}\le \| \sum_{V\in\cV_3} 1_{V_\de}\|_p\le C_\e \de^{-\frac{k(d-k)}{p'}-\e}\big(\sum_{V\in\cV_3}|V_\de|\big)^{1/p}.    \]
Noting \eqref{lowerbd2}, this implies 
\[ | E_\de |\gtrsim \de^{k(d-k)+p'\e}\big(\sum_{V\in\cV_3}|V_\de|\big). \]
Combined with $\sum_{V\in\cV_3}|V_\de|=\#\cV_3 \de^{n-k}\gtrsim \de^{O(\e)-(k+1)(d-k)-\beta} \de^{n-k}$ and \eqref{lowerbd}, we obtain that
\[ \dim X_1\ge d+\beta-O(\e) \]
as $\de$ goes to $0$.
\end{proof}

\bibliographystyle{abbrv}
\bibliography{bibli}

\end{document}